\documentclass[12pt]{amsart}
\usepackage{amssymb,amsmath,amsfonts,amsthm,bbm,mathrsfs,times,graphicx,color,comment,mathabx,enumerate}
\usepackage[utf8]{inputenc}
\usepackage{tikz}
\usetikzlibrary{arrows,shapes}
\usetikzlibrary{fadings}
\usetikzlibrary{intersections}
\usetikzlibrary{decorations.pathreplacing}

\usepackage{scalerel}
\DeclareMathOperator{\re}{\mathbb{R}e}
\DeclareMathOperator{\im}{\mathbb{I}m}

\renewcommand{\i}{\mathrm i}

\DeclareMathOperator{\Div}{div}
\DeclareMathOperator{\curl}{curl}

\newcommand{\INF}{{\infty}}

\newcommand{\tta}{\theta}

\newcommand{\OM}{\Omega}

\newcommand{\sph}{{{\mathbf S}^ 1}}

\newcommand{\del}{\partial}

\newcommand{\Gam}{\varGamma}
\newcommand{\ol}{\overline}
\newcommand{\ds}{\displaystyle}
\newcommand{\dba}{\overline{\partial}}

\newcommand{\BR}{\mathbb{R}}
\newcommand{\BC}{\mathbb{C}}

\newcommand{\BZ}{\mathbb{Z}}

\newcommand{\bu}{{\bf u}}
\newcommand{\bv}{{\bf v}}
\newcommand{\bw}{{\bf w}}

\newcommand{\bg}{{\bf g}}

\newcommand{\bF}{{\bf F}}

\newcommand{\bzero}{\mathbf 0}

\newcommand{\btheta}{\boldsymbol \theta}
\newcommand{\balpha}{\boldsymbol \alpha}
\newcommand{\bbeta}{\boldsymbol \beta}

\newcommand{\brho}{\mbox{\boldmath{$\rho$}}}

\newcommand{\jpj}{\langle j\rangle}

\newcommand{\B}{\mathcal{B}}

\newcommand{\HT}{\mathcal{H}}

\newcommand{\lnorm}[1]{ \left\| #1 \right\|}

\newcommand\numberthis{\addtocounter{equation}{1}\tag{\theequation}}

\newcommand{\Laplace}{\mathop{}\!\mathbin\bigtriangleup}
\newtheorem{theorem}{Theorem}[section]
\newtheorem{prop}{Proposition}[section]
\newtheorem{lemma}{Lemma}[section]

\newtheorem{remark}{Remark}[section]

\title[An inverse source problem for linearly anisotropic radiative sources ]
{An inverse source problem for linearly anisotropic radiative sources in absorbing and scattering medium}

\begin{document}
	\date{\today}
	
	\author{David Omogbhe}
	\address{Faculty of Mathematics, University of Vienna, Oskar-Morgenstern-Platz 1, 1090 Vienna, Austria}
	\email{david.omogbhe@univie.ac.at}
	\author{Kamran Sadiq}
	\address{Johann Radon Institute for Computational and Applied Mathematics (RICAM), Altenbergerstrasse 69, 4040 Linz, Austria}
	\email{kamran.sadiq@ricam.oeaw.ac.at}
	
	
	
	
	\maketitle
	
	\begin{abstract}
		We consider in a two dimensional absorbing and scattering medium, an inverse source problem in the stationary radiative transport, where the source is linearly anisotropic. 
		The medium has an anisotropic scattering property that is neither negligible nor large enough for the diffusion approximation to hold.
		The attenuating and scattering properties of the medium are assumed known. 
		For scattering kernels of finite Fourier content in the angular variable, we show how to  recover the anisotropic radiative sources from boundary measurements.  The approach is based on the Cauchy problem for a Beltrami-like equation associated with $A$-analytic maps. 
		As an application, we determine necessary and sufficient conditions for the data coming from two different sources to be mistaken for each other.
	\end{abstract}
	\section{Introduction}
	
	In this work, we consider an inverse source problem for stationary radiative transfer (transport)  \cite{cercignani,chandrasekhar}, in a two-dimensional bounded, strictly convex domain $\OM\subset \BR^2$, with boundary $\Gamma$. The stationary radiative transport models the linear transport of particles through a medium and includes absorption and scattering phenomena.  
	In the steady state case, when generated solely by a linearly anisotropic source $f(z, \btheta)=f_0(z) + \btheta \cdot \bF(z)$ inside $\OM$,  the density  $u(z,\btheta)$ of particles at $z$  traveling in the direction $\btheta$ 
	solves the stationary radiative transport boundary value problem 
	\begin{equation}\label{TransportScatEq1}
		\begin{aligned} 
			&\btheta\cdot\nabla u(z,\btheta) +a(z,\btheta) u(z,\btheta) - \int_{\sph} k(z,\btheta,\btheta')u(z,\btheta') d\btheta' = f(z,\btheta),
			\quad (z,\btheta)\in \OM\times\sph,\\
			&u\lvert_{\Gamma_-}=0,
		\end{aligned}
	\end{equation}
	where the function $a(z, \btheta)$ is the medium capability of absorption per unit path-length at $z$ moving in the direction $\btheta$ called the attenuation coefficient, function $k(z,\btheta,\btheta')$ is the scattering coefficient which accounts for particles from an arbitrary direction $\btheta'$ which scatter in the direction $\btheta$ at a point $z$, and $\Gam_{-} :=\{(\zeta,\btheta)\in \Gam \times\sph:\, \nu(\zeta)\cdot\btheta<0 \}$ is the incoming unit tangent sub-bundle of the boundary, with $\nu(\zeta)$ being the outer unit normal at $\zeta \in \Gam$. 
	The attenuation and scattering coefficients are assumed known real valued functions.
	The boundary condition in \eqref{TransportScatEq1} indicates that no radiation is coming from outside the domain. Throughout, the measure  $d \btheta$ on the unit sphere $\sph$ is normalized to $ \int_{\sph} d\btheta =1$. 

    Under various assumptions, e.g., \cite{dautrayLions4, choulliStefanov99,  anikonov02,mokhtar, balTamasan07},  the (forward) boundary value problem  \eqref{TransportScatEq1} is known to have a unique solution,  with a general result in  \cite{stefanovUhlmann08} showing that, for an open and dense set of coefficients $a\in C^2(\ol\Omega\times\sph)$ and $k\in C^2(\ol\Omega\times\sph\times\sph)$, 
	the boundary value problem  \eqref{TransportScatEq1} has a unique solution $u\in L^2(\Omega\times\sph)$ for any $f\in L^2(\Omega\times\sph)$. 
	In \cite{fujiwaraSadiqTamasan21b}, it is shown that for attenuation merely \emph{once} differentiable,  $a\in C^1(\ol\Omega\times\sph)$ and $k\in C^2(\ol\Omega\times\sph\times\sph)$, 
	the boundary value problem  \eqref{TransportScatEq1} has a unique solution $u\in L^p(\Omega\times\sph)$ for any $f\in L^p(\Omega\times\sph)$, $p >1$.
	Moreover,  uniqueness result of  the forward problem \eqref{TransportScatEq1} are also establish in weighted $L^p$ spaces in \cite{eggerSchlottbom}.
	In our reconstruction method here, some of our arguments require solutions $u\in C^{2,\mu}(\ol\OM\times\sph)$, $\frac{1}{2}<\mu<1$. We revisit the arguments in  \cite{stefanovUhlmann08,fujiwaraSadiqTamasan21b} and show that such a regularity can be achieved for sources $f\in W^{3,p}(\Omega\times\sph)$, $p>4$; see Theorem \ref{u_reg_Wp} (iii) below.
	
	For a given medium, i.e., $a$ and $k$ both known, we consider the inverse problem of determining the scalar field $f_0$, and the vector field $\bF$  from measurements $g_{\scaleto{f_0,\bF}{5.75pt}}$ of exiting radiation on $\Gamma$,
	\begin{align}\label{data}
		u|_{\Gamma_+} = g_{\scaleto{f_0,\bF}{4.75pt}},
	\end{align}
	where  $\Gamma_+:=\{(z,\btheta)\in \Gamma \times\sph:\,\nu(z)\cdot\btheta>0 \}$ 
	is the outgoing unit tangent sub-bundle of the boundary, with $\nu(\zeta)$ being the outer unit normal at $\zeta \in \Gam$. 
	
	For anisotropic sources the problem has non-uniqueness \cite{vladimirBook, tamasan07}.
	One of  our main result, Theorem \ref{Th:LuFs} shows that from boundary measurement data $ g_{\scaleto{f_0,\bF}{4.75pt}}$, 
	one can only recover the part of the linear anisotropic source $f = f_0 +\btheta \cdot \bF$; in particular, only the  solenoidal part $ \bF^s$ of the vector field $ \bF$ is recovered
	inside the domain. However, in Theorem \ref{main_divfree}, if one know apriori that the source $\bF$ is divergence-free, then from the data $ g_{\scaleto{f_0,\bF}{4.75pt}}$, one can recover both isotropic field $f_0$ and the vector field $\bF$ inside the domain. 
	Moreover, instead of apriori information of the divergence-free source  $\bF$, if one has the additional data $ g_{\scaleto{f_0,\bzero}{4.75pt}}$ information along with the data  $ g_{\scaleto{f_0,\bF}{4.75pt}}$, then in Theorem \ref{Th_2data}, one can recover both sources $f_0$ and $\bF$ under subcritical assumption of the medium. One of the main crux in our reconstruction method is the observation that any finite Fourier content in the angular variable of the scattering kernel splits the problem into an infinite system of non-scattering case and a boundary value problem for a finite elliptic system. The role of the finite Fourier content has been independently recognized in \cite{fujiwaraSadiqTamasan19} and  \cite{monardBalPreprint}.
	
	
	The inverse source problem above has applications in medical imaging: In a non-scattering ($k=0$) and non-attenuating ($a=0$) medium the problem is mathematically equivalent to the one occurring in classical computerized $X$-ray tomography (e.g., \cite{bukhgeimBook,nattererBook}).
		In the absorbing non-scattering medium,
	 such a problem (with only isotropic source $f=f_0$), appears in Positron/Single Photon Emission Tomography (PET/SPECT) \cite{nattererBook,nattererWubbeling}, and (with $f_0 = 0$ and $f=\btheta\cdot\bF$), appears in Doppler Tomography \cite{		nattererWubbeling, nattererBook,sparSLP95}.
	    For applications in scattering media the inverse source problem
	 	formulated here is the two dimensional version of the corresponding three dimensional problem occurring in imaging techniques such as Bioluminescence tomography and Optical Molecular Imaging, see  \cite{YiSanchezMccormick92,han,klose} and references therein.

	In Section \ref{sec:forward}, we remark on the existence and regularity of the forward boundary value problem. The results in Section  \ref{sec:forward} consider both attenuation coefficient and scattering kernel in general setting. 
	
	In this work, except for the results in Section \ref{sec:forward}, 
	the attenuation coefficient are assumed isotropic  $a=a(z)$, and that the scattering kernel $k(z,\btheta,\btheta')=k(z,\btheta\cdot\btheta')$ depends polynomially on the angle between the directions.  Moreover, the functions $a,k$ and the source $f$ are assumed real valued.

	In Section \ref{sec:prelim}, we recall some basic properties of $A$-analytic theory, and in Section \ref{sec:reconstruct} we provide the reconstruction method for the full (part) of the linearly anisotropic source. Our approach is based on the Cauchy problem for a Beltrami-like equation associated with $A$-analytic maps in the sense of Bukhgeim \cite{bukhgeimBook}.  The $A$-analytic approach developed in \cite{bukhgeimBook}  treats the  non-attenuating case, and  the absorbing but non-scattering case is treated in \cite{ABK}.  The original idea of Bukhgeim from the absorbing non-scattering media \cite{bukhgeimBook,ABK} to the absorbing and scattering media has been extended in \cite{fujiwaraSadiqTamasan19,fujiwaraSadiqTamasan21b}. In here we extend the results  in \cite{fujiwaraSadiqTamasan19,fujiwaraSadiqTamasan21b} to linear anisotropic sources.

In Section \ref{sources_mistaken}, the method used will explain when the data coming from two different linear anisotropic field sources can be mistaken for each other.
	\section{Some remarks on the existence and regularity of the forward problem}\label{sec:forward}
	
	In this section, we revisit the arguments in \cite{stefanovUhlmann08, fujiwaraSadiqTamasan21b},  and remark on the well posedness in $L^p(\OM\times\sph)$ of the boundary value problem \eqref{TransportScatEq1}.
	Adopting the notation in \cite{stefanovUhlmann08,fujiwaraSadiqTamasan21b}, we consider the operators
	\begin{align*} 
		[T_1^{-1} g ] (x, \btheta) &= \int_{-\INF}^{0} e^{- \int_{s}^{0} a(x +t \btheta, \btheta ) dt} g(x + s \btheta,\btheta) ds,
		\: \text{and} \; 
		[K g ] (x, \btheta) = \int_{\sph} k(x,\btheta,\btheta')g(x,\btheta') d\btheta',
	\end{align*}
	where the intervening functions are extended by 0 outside $\OM$.
	
	Using the above operators, the problem \eqref{TransportScatEq1} can be rewritten as 
	\begin{align}\label{transportEq_equiv}
		(I - T_1^{-1} K) u &= T_1^{-1} f, \qquad u\lvert_{\Gamma_-}=0.
	\end{align}
	If the operator $I- T_1^{-1}K$ is invertible, then the problem \eqref{transportEq_equiv} is uniquely solvable, and has  the form 
	$u = (I - T_1^{-1} K)^{-1} T_1^{-1} f $.
	By using the formal expansion
	\begin{equation}\label{u_decomp}
		\begin{aligned}
			u &= T_1^{-1} f+ T_1^{-1} K T_1^{-1} f +T_1^{-1} (KT_1^{-1}K)[I - T_1^{-1} K]^{-1} T_1^{-1} f,
		\end{aligned}
	\end{equation}
	the well posed-ness in $L^p(\Omega\times\sph)$ of the (forward)  boundary value problem \eqref{TransportScatEq1} is reduced to showing that the operator  $I- T_1^{-1}K$  is invertible in   $L^p(\OM\times \sph)$. 
	
	We recall some results in  \cite{fujiwaraSadiqTamasan21b}.
	\begin{prop}\cite[Proposition 2.1]{fujiwaraSadiqTamasan21b}\label{KToneK_prop}
		Let $a\in C^1(\ol\Omega\times\sph)$ and $k \in C^{2}(\ol\Omega\times\sph\times\sph)$. Then the operator
		\begin{align} \label{KToneK_reg}
			K T_1^{-1} K: L^p(\OM \times \sph) \longrightarrow W^{1,p}(\Omega\times\sph) \mbox{ is bounded},
			\; 1<p<\infty. 
		\end{align}
	\end{prop}   
	
	The following simple result is useful.
	\begin{lemma}  \cite[Lemma 2.2]{fujiwaraSadiqTamasan21b}\label{lem:12} Let $X$ be a Banach space and $A:X\to X$ be bounded. Then $I \pm A$ have bounded inverses in $X$, if and only if $I-A^2$ has a bounded inverse in $X$.
	\end{lemma}
	
	For $\lambda\in\BC$, we note that $\ds (T_1^{-1} (\lambda K))^2 = \lambda^2 T_1^{-1} (KT_1^{-1}K)$. By Proposition \ref{KToneK_prop}, the operator $\ds (T_1^{-1} (\lambda K))^2$  is compact for any $\lambda \in\BC$.  By Lemma \ref{lem:12}, if the operator $I- (T_1^{-1}(\lambda K))^2$ is invertible in  $L^p(\OM\times \sph)$, then the operator $I- T_1^{-1}(\lambda K)$ is invertible in  $L^p(\OM\times \sph)$. Since $\ds I- (T_1^{-1} (\lambda K))^2$ is invertible for $\lambda$ in a neighborhood of $0$, an application of the analytic Fredholm alternative in Banach spaces, e.g., \cite[Theorem VII.4.5]{dunfordScwartz}, yields the following result.
	\begin{theorem}   \cite[Theorem 2.1]{fujiwaraSadiqTamasan21b}\label{analytic_Fredholm} 
		Let $p>1$, $a\in C^{1}(\ol\Omega\times\sph)$, and $k\in C^{2}(\ol\Omega\times\sph\times\sph)$. At least one of the following statements is true.
		
		(i) $\ds I - T_1^{-1} K$ is invertible in $ L^p(\OM\times \sph)$.
		
		(ii) there exists $\epsilon>0$ such that $\ds I - T_1^{-1} (\lambda K)$ is invertible in $ L^p(\OM\times \sph)$, for any 
		$0< |\lambda-1|<\epsilon$.
	\end{theorem}
	
	The regularity of the solution $u$ of \eqref{TransportScatEq1} increases with the regularity of $f$ as follows.
	\begin{theorem} \label{u_reg_Wp}
		Consider the boundary value problem \eqref{TransportScatEq1} with $a\in C^3(\ol\OM\times\sph)$. For $p>1$, let $k\in C^{3}(\ol\Omega\times\sph\times\sph)$ be such that $\ds I - T_1^{-1}K$ is invertible in $L^p(\OM\times\sph)$, and let $u\in L^p(\OM\times\sph)$ in \eqref{u_decomp} be the solution  of \eqref{TransportScatEq1}.
		\begin{enumerate}
			\item[(i)] If  $f \in W^{1,p}(\OM \times \sph)$, then $u \in W^{1,p}(\OM \times \sph)$. 
			\item[(ii)]  If  $f \in W^{2,p}(\OM \times \sph)$, then $u \in W^{2,p}(\OM \times \sph)$. 
			\item[(iii)]  If  $f \in W^{3,p}(\OM \times \sph)$, then $u \in W^{3,p}(\OM \times \sph)$.
		\end{enumerate}
		
	\end{theorem}
	\begin{proof}
		
		(i) We consider the regularity of the solution $u$ of \eqref{TransportScatEq1} term by term as in \eqref{u_decomp}: 
		\begin{align*}
			u 
			&=T_1^{-1} f+ T_1^{-1} K T_1^{-1} f +T_1^{-1} [KT_1^{-1}K ](I - T_1^{-1} K)^{-1} T_1^{-1} f.
		\end{align*}
		It is easy to see that the operator  $T_1^{-1}$  preserve the space $W^{1,p}(\Omega\times\sph)$, 
		and also the operator $K$  preserve the space $W^{1,p}(\Omega\times\sph)$,
		so that the first two terms, $T_1^{-1} f$ and $T_1^{-1} K T_1^{-1} f$, both  belong to $W^{1,p}(\Omega\times\sph)$. Moreover, $(I - T_1^{-1} K)^{-1 } T_1^{-1}f \in  L^p(\Omega\times\sph) $,
		and now, by using Proposition \ref{KToneK_prop}, the last term is also belong in $W^{1,p}(\OM\times\sph)$. 
		
		(ii)  We define the following operators
		\begin{equation}\label{T0tilde_Kj_defn}
			\begin{aligned}
				&T_0^{-1}u(x,\btheta):=\int_{-\infty}^0 u(x+t\btheta,\btheta)dt,  &K_{\xi_j}u(x,\btheta):=\int_\sph \frac{\del k}{\del \xi_j}(x,\btheta,\btheta')u(x,\btheta')d\btheta',  \\
				&\widetilde{T}_{0,j}^{-1}u(x,\btheta):=\int_{-\infty}^0 u(x+t\btheta,\btheta)t^jdt, &
				K_{\eta_i \xi_j}u(x,\btheta):	=\int_\sph \frac{\del^2 k}{\del \eta_i \del \xi_j }(x,\btheta,\btheta')u(x,\btheta')d\btheta', 
			\end{aligned}
		\end{equation} where $\eta_i = \{x_i, \theta_i\} $ and $\xi_j = \{x_j, \theta_j\}$ for $i , j =1,2$.\\
		It is easy to see that $T_0^{-1}, \widetilde{T}_{0,j}^{-1},  K_{\xi_j}$ and $K_{\eta_i \xi_j}$ preserve $W^{1,p}(\OM\times\sph)$.

		By evaluating the radiative transport equation in \eqref{TransportScatEq1} at $x + t \btheta$ and integrating in $t$ from $-\INF$ to $0$, the boundary value problem \eqref{TransportScatEq1} with zero incoming fluxes is equivalent to the integral equation:
		\begin{align}\label{TransportScatEq1_Equiv}
			u+ T_0^{-1} (au)- T_0^{-1}Ku=T_0^{-1}f.
		\end{align}

		According to part (i), for $f\in W^{1,p}(\OM\times\sph)$, the solution  $u\in W^{1,p}(\OM\times\sph)$, 
		and so $u_{x_j}\in L^p(\OM\times\sph)$. In particular $u_{x_j}$ solves the integral equation:
		\begin{align}\label{TransportScatEq1_Equiv_delxU}
			u_{x_j} + T_0^{-1} (au_{x_j})- T_0^{-1}Ku_{x_j}= - T_0^{-1} (a_{x_j}u)+ T_0^{-1}K_{x_j}u + T_0^{-1}f_{x_j}.
		\end{align}Moreover, since $a\in C^2(\ol\OM\times\sph)$, $k\in C^{2}(\ol\Omega\times\sph\times\sph)$, and $f \in W^{2,p}(\OM \times \sph)$, the right-hand-side of  \eqref{TransportScatEq1_Equiv_delxU} lies in $W^{1,p}(\OM \times \sph)$. By applying part (i) above, we get that the unique solution to \eqref{TransportScatEq1_Equiv_delxU}
		\begin{align}\label{xfirst}
			u_{x_j}  \in W^{1,p}(\OM \times \sph), \; j=1,2.
		\end{align}

		For $f\in W^{1,p}(\OM\times\sph)$, also according to part (i),  $u_{\theta_j}\in L^p(\OM\times\sph)$. In particular $u_{\theta_j}$ is the unique solution of the integral equation
		\begin{align}\label{TransportScatEq1_Equiv_DthetaU}
			u_{\theta_j}+ T_0^{-1}(a u_{\theta_j}) =  -\widetilde{T}_{0,1}^{-1} (au_{x_j})- T_0^{-1}(a_{\theta_j}u) -\widetilde{T}_{0,1}^{-1} (a_{x_j}u)  +T_0^{-1}K_{\theta_j} u +\widetilde{T}_{0,1}^{-1}K_{x_j}u + T_0^{-1} f_{\theta_j},
		\end{align} which is of the type \eqref{TransportScatEq1_Equiv} with $K=0$. Moreover, since $f \in W^{2,p}(\OM \times \sph)$, and, according to  \eqref{xfirst}, $u_{x_j}  \in W^{1,p}(\OM \times \sph), \; j=1,2$, the right-hand-side of  \eqref{TransportScatEq1_Equiv_DthetaU} lies in $W^{1,p}(\OM \times \sph)$. Again, by applying part (i), we get 
		$$u_{\theta_j} \in W^{1,p}(\OM \times \sph), \; j=1,2.$$ Thus, $u \in W^{2,p}(\OM \times \sph)$.
		
		(iii) 
		For $f\in W^{2,p}(\OM\times\sph)$, according to part (ii), $u_{x_j}, u_{\theta_j} \in W^{1,p}(\OM\times\sph)$, and  $u_{x_i x_j}\in L^p(\OM\times\sph)$. In particular $u_{x_i x_j}$ is the unique solution of the integral equation
		\begin{equation}\label{TransportScatEq1_Equiv_delxxU}
			\begin{aligned}
				u_{x_ix_j}+ T_0^{-1}(au_{x_ix_j})-T_0^{-1}(Ku_{x_ix_j}) &= T_0^{-1}f_{x_ix_j}-T_0^{-1}(a_{x_j}u_{x_i})-T_0^{-1}(a_{x_ix_j}u)+T_0^{-1}(K_{x_j}u_{x_i})\\
				&\quad + T_0^{-1}(K_{x_ix_j}u)-T_0^{-1}(a_{x_i}u_{x_j})-T_0^{-1}(K_{x_i}u_{x_j}).
			\end{aligned}
		\end{equation}
		Moreover, since $a\in C^3(\ol\OM\times\sph)$, $k\in C^{3}(\ol\Omega\times\sph\times\sph)$, and $f \in W^{3,p}(\OM \times \sph)$, the right-hand-side of  \eqref{TransportScatEq1_Equiv_delxxU} lies in $W^{1,p}(\OM \times \sph)$. By applying part (i) above, we get that the unique solution to \eqref{TransportScatEq1_Equiv_delxxU}
		\begin{align}\label{u_xx}
			u_{x_i x_j}  \in W^{1,p}(\OM \times \sph), \; i,j=1,2.
		\end{align}

		For $f\in W^{2,p}(\OM\times\sph)$, also according to part (ii), $u_{x_j}, u_{\theta_j} \in W^{1,p}(\OM\times\sph)$, and  $u_{\theta_i \theta_j}\in L^p(\OM\times\sph)$. In particular $u_{\theta_i \theta_j}$ is the unique solution of the integral equation
		\begin{align*}
			u_{{\theta}_i{\theta}_j}+T_0^{-1}(au_{{\theta}_i{\theta}_j}) &= T_0^{-1}(f_{{\theta}_i{\theta}_j}) - \widetilde{T}_{0,2}^{-1}(a_{x_i}u_{x_j})-\widetilde{T}_{0,1}^{-1}(a_{x_i} u_{\theta_j} )-\widetilde{T}_{0,1}^{-1}( a_{\theta_i} u_{x_j}) \\
			&
			\quad -T_0^{-1}( a_{\theta_i} u_{\theta_j} ) -\widetilde{T}_{0,2}^{-1}(a_{x_j}u_{x_i})- \widetilde{T}_{0,1}^{-1}(a_{x_j}  u_{\theta_i})-\widetilde{T}_{0,2}^{-1}(a_{x_ix_j}u)\\
			&\quad - \widetilde{T}_{0,1}^{-1}( a_{x_j \theta_i} u) -\widetilde{T}_{0,1}^{-1}(  a_{\theta_j} u_{x_i})- T_0^{-1}( a_{\theta_j} u_{\theta_i} )  
			-\widetilde{T}_{0,1}^{-1}( a_{\theta_i \theta_j} u)\\ \numberthis \label{TransportScatEq1_Equiv_DthetathetaU}
			&\quad - \widetilde{T}_{0,1}^{-1}(K_{{\theta}_j}u_{x_i}) -T_0^{-1}(K_{{\theta}_i{\theta}_j}u) - \widetilde{T}_{0,2}^{-1}(Ku_{x_ix_j})- \widetilde{T}_{0,1}^{-1}(K_{{\theta}_i}u_{x_j}),
		\end{align*}
		which is of the type \eqref{TransportScatEq1_Equiv} with $K=0$. 
		
		Moreover, since $f \in W^{3,p}(\OM \times \sph)$, and, according to  \eqref{u_xx}, $u_{x_i x_j}  \in W^{1,p}(\OM \times \sph), \; j=1,2$, the right-hand-side of  \eqref{TransportScatEq1_Equiv_DthetathetaU} lies in $W^{1,p}(\OM \times \sph)$. Again, by applying part (i), we get 
		\begin{align}\label{u_thetatheta}
			u_{\theta_i \theta_j} \in W^{1,p}(\OM \times \sph), \; i,j=1,2.
		\end{align}
		
		For $f\in W^{2,p}(\OM\times\sph)$, also according to part (ii), $u_{x_i} u_{\theta_j} \in  L^p(\OM\times\sph)$. In particular $u_{x_i \theta_j}$ is the unique solution of the integral equation
		
		\begin{align*}
			u_{x_i{\theta}_j}+T_0^{-1}(au_{x_i{\theta}_j})- T_0^{-1}(Ku_{x_j{\theta}_i}) &= T_0^{-1} (f_{x_j{\theta}_i})-\widetilde{T}_{0,1}^{-1}(a_{x_i}u_{x_j})- T_0^{-1}(a_{{\theta}_i} u_{x_j})
			\\
			&\quad 
			-T_0^{-1}(a_{x_j{\theta}_i}u) 
			- \widetilde{T}_{0,1}^{-1}(a_{x_j}u_{x_i})-T_0^{-1}(u_{{\theta}_i}a_{x_j})
			\\ \numberthis \label{TransportScatEq1_Equiv_DxthetaU}
			&\quad + \widetilde{T}_{0,1}^{-1}(K_{x_j}u_{x_i}) 
			+ T_0^{-1}(K_{{\theta}_i}u_{x_j}) 
			+T_0^{-1}(K_{x_j{\theta}_i}u),
		\end{align*}
		which is of the type \eqref{TransportScatEq1_Equiv}. 
		Moreover, since $a\in C^3(\ol\OM\times\sph)$, $k\in C^{3}(\ol\Omega\times\sph\times\sph)$, and $f \in W^{3,p}(\OM \times \sph)$, the right-hand-side of  \eqref{TransportScatEq1_Equiv_DxthetaU} lies in $W^{1,p}(\OM \times \sph)$.
		Again, by applying part (i), we get 
		\begin{align}\label{u_xtheta}
			u_{x_i \theta_j} \in W^{1,p}(\OM \times \sph), \; i,j=1,2.
		\end{align}
		
		From \eqref{u_xx}, \eqref{u_thetatheta}, and \eqref{u_xtheta}, we get $u\in W^{3,p}(\OM \times \sph)$.
		
	\end{proof}
	
	We remark that for Theorem \ref{u_reg_Wp} part (i) we only need $a\in C^1(\ol\OM\times\sph)$ and $k\in C^{2}(\ol\Omega\times\sph\times\sph)$, and  we only require $a\in C^2(\ol\OM\times\sph)$ and $k\in C^{2}(\ol\Omega\times\sph\times\sph)$  for Theorem \ref{u_reg_Wp} part (ii).  We also refer to
	\cite[Theorem 2.2]{fujiwaraSadiqTamasan21b} for part (i) and (ii) of Theorem \ref{u_reg_Wp}.
	Moreover, in a similar fashion, one can show that under sufficiently increased regularity of $a$ and $k$, the solution $u$ of \eqref{TransportScatEq1} belong to $u\in W^{m,p}(\OM \times \sph)$ for $\BZ \ni m \geq1$, provided $f \in W^{m,p}(\OM \times \sph)$.
	\vspace{-0.42cm}
	
	\section{Ingredients from  $A$-analytic theory}\label{sec:prelim}
	In this section we briefly introduce the properties of $A$-analytic maps needed later, and introduce notation. We recall some of the existing results and concepts used in our reconstruction method.

	
	For $0<\mu<1$, $p=1,2$, we consider the  Banach spaces:
	\begin{equation}\label{spaces}
		\begin{aligned} 
			l^{1,p}_{\INF}(\Gam) &:= \left \{ \bg= \langle g_{0}, g_{-1}, g_{-2},...\rangle\; : \lnorm{\bg}_{l^{1,p}_{\INF}(\Gam)}:= \sup_{\xi \in \Gam}\sum_{j=0}^{\INF}  \jpj^p \lvert g_{-j}(\xi) \rvert < \INF \right \},\\
			C^{\mu}(\Gam; l_1) &:= \left \{ \bg= \langle g_{0}, g_{-1}, g_{-2},...\rangle:
			\sup_{\xi\in \Gam} \lVert \bg(\xi)\rVert_{\ds l_{1}} + \underset{{\substack{
						\xi,\eta \in \Gam \\
						\xi\neq \eta } }}{\sup}
			\frac{\lVert \bg(\xi) - \bg(\eta)\rVert_{\ds l_{1}}}{|\xi - \eta|^{ \mu}} < \INF \right \}, \\
			Y_{\mu}(\Gam) &:= \left \{ \bg: \bg \in  l^{1,2}_{\INF}(\Gam) \; \text{and} \;
			\underset{{\substack{
						\xi,\eta \in \Gam \\
						\xi\neq \eta } }}{\sup} \sum_{j=0}^{\INF}  \jpj 
			\frac{\lvert g_{-j}(\xi) - g_{-j}(\eta)\rvert }{|\xi - \eta|^{ \mu}} < \INF \right \},
		\end{aligned}
	\end{equation} where, for brevity, we use the notation $\jpj=(1+|j|^2)^{1/2}$.
	Similarly,  we consider $ C^{\mu}(\ol \OM; l_1) $, and $ C^{\mu}(\ol \OM; l_\INF) $.
	
	For $z=x_1+\i x_2$, we consider the Cauchy-Riemann operators
	\begin{align}\label{dbar_op} 
		\ol{\del} = \left( \del_{x_{1}}+\i \del_{x_{2}} \right) /2 ,\quad \del = \left( \del_{x_{1}}- \i \del_{x_{2}} \right) /2.
	\end{align}	
	
	
	A sequence valued map $\OM \ni z\mapsto  \bv(z): = \langle v_{0}(z), v_{-1}(z),v_{-2}(z),... \rangle$ in $C(\ol\OM;l_\INF)\cap C^1(\OM;l_\INF)$
	is called {\em $L^2$-analytic} (in the sense of Bukhgeim), if
	\begin{equation}\label{Aanalytic}
		\ol{\del} \bv (z) + L^2 \del \bv (z) = 0,\quad z\in\OM,
	\end{equation}
	where $L$ is the left shift operator $\ds L \langle v_{0}, v_{-1}, v_{-2}, \cdots  \rangle =  \langle v_{-1}, v_{-2},  \cdots \rangle, $ and
	$L^{2}=L\circ L$. 
	
	Bukhgeim's original  theory \cite{bukhgeimBook}  shows that solutions of \eqref{Aanalytic},  satisfy a Cauchy-like integral formula,
	\begin{align}\label{Analytic}
		\bv (z) = \B [\bv \lvert_{\Gam}](z), \quad  z\in\OM,
	\end{align} where $\B$ is 
	the Bukhgeim-Cauchy operator  acting on $\bv \lvert_{\Gam}$. We use the formula in \cite{finch}, where
	$\B$ is defined component-wise for $n\geq 0$ by
	\begin{equation} \label{BukhgeimCauchyFormula}
		\begin{aligned} 
			(\B \bv)_{-n}(z) &:= \frac{1}{2\pi \i} \int_{\Gam}
			\frac{ v_{-n}(\zeta)}{\zeta-z}d\zeta  
			+ \frac{1}{2\pi \i}\int_{\Gam} \left \{ \frac{d\zeta}{\zeta-z}-\frac{d \ol{\zeta}}{\ol{\zeta}-\ol{z}} \right \} \sum_{j=1}^{\infty}  
			v_{-n-2j}(\zeta)
			\left( \frac{\ol{\zeta}-\ol{z}}{\zeta-z} \right) ^{j},\; z\in\OM.
		\end{aligned}
	\end{equation}
	
	The theorems below comprise some results in  \cite{sadiqTamasan01,sadiqTamasan02}. For the proof of the theorem below we refer to 
	\cite[Proposition 2.3]{sadiqTamasan02}.
	\begin{theorem}\label{BukhgeimCauchyThm}
		Let $0<\mu<1$, and let $\B$ be the Bukhgeim-Cauchy operator in \eqref{BukhgeimCauchyFormula}. 
		
		If $\bg = \langle g_{0}, g_{-1}, g_{-2},...\rangle\in Y_{\mu}(\Gam)$ for $\mu>1/2$, then $ \bv :=\B \bg\in C^{1,\mu}(\OM;l_1)\cap C^{\mu}(\ol \OM;l_1)\cap C^2(\OM;l_\infty)$ is $L^2$-analytic in $\OM$.
	\end{theorem}

	Similar to the analytic maps, the traces of $L^2$-analytic maps  on the boundary must satisfy some constraints, which can be expressed in terms of a corresponding Hilbert-like transform introduced in  \cite{sadiqTamasan01}. More precisely, the Bukhgeim-Hilbert transform $\HT$ acting on  $\bg$, 
	\begin{align}\label{boldHg}
		\Gam \ni z\mapsto  (\HT \bg)(z)& = \langle (\HT \bg)_{0}(z), (\HT \bg)_{-1}(z),(\HT \bg)_{-2}(z),... \rangle
	\end{align}
	is defined component-wise for $n\geq 0$ by
	
	\begin{equation} \label{BHtransform}
		\begin{aligned} 
			(\HT \bg)_{-n}(z)&=\frac{1}{\pi }\int_\Gam \frac{ g_{-n}(\zeta)}{\zeta-z}d\zeta
			+ \frac{1}{\pi }\int_{\Gam} \left \{ \frac{d\zeta}{\zeta-z}-\frac{d \ol{\zeta}}{\ol{\zeta}-\ol{z}} \right \} \sum_{j=1}^{\infty}  
			g_{-n-2j}(\zeta)
			\left( \frac{\ol{\zeta}-\ol{z}}{\zeta-z} \right) ^{j},\; z\in \Gam,
		\end{aligned}
	\end{equation}
	and we refer to \cite{sadiqTamasan01} for its mapping properties. 
	
	The following result recalls the necessary and sufficient conditions for a sufficiently regular map to be the boundary value of an  $L^2$-analytic function.
	\begin{theorem}\label{NecSuf_BukhgeimHilbert_Thm}
		Let $0<\mu<1$, and $\B$ be the Bukhgeim-Cauchy operator in \eqref{BukhgeimCauchyFormula}.  \\
		Let $\bg = \langle g_{0}, g_{-1}, g_{-2},...\rangle\in Y_{\mu}(\Gam)$ for $\mu>1/2$ be defined on the boundary $\Gamma$, 
		and let $\HT$ be the Bukhgeim-Hilbert transform acting on $\bg$ as in \eqref{BHtransform}.
		
		(i) If $\bg$  is the boundary value of an $L^2$-analytic function, 
		then $\HT \bg\in C^{\mu}(\Gam;l_1)$ and satisfies 
		\begin{align} \label{NecSufEq}
			(I+ \i \HT) \bg = {\bf {0}}.
		\end{align}
		(ii)  If $\bg$  satisfies \eqref{NecSufEq}, then there exists an $L^2$-analytic function $ \bv :=\B \bg \in C^{1,\mu}(\OM;l_1)\cap C^{\mu}(\ol \OM;l_1)\cap C^2(\OM;l_\infty)$, such that
		\begin{align}\label{gdata_defn}
			\bv \lvert_{\Gam} = \bg.
		\end{align}
	\end{theorem}
	For the proof  of  Theorem \ref{NecSuf_BukhgeimHilbert_Thm}  we refer to \cite[Theorem 3.2,  Corollary 4.1, and Proposition 4.2]{sadiqTamasan01} and \cite[Proposition 2.3]{sadiqTamasan02}.
	
	Another ingredient, in addition to $L^2$-analytic maps,  consists in the one-to-one relation between solutions
	$ \bu: = \langle u_{0}, u_{-1},u_{-2},... \rangle $
	satisfying
	\begin{align}\label{beltrami}
		\dba\bu +L^2 \del\bu+ aL\bu = \bzero,
	\end{align}
	and the $L^2$-analytic map $\bv = \langle v_{0}, v_{-1},v_{-2},... \rangle $ satisfying \eqref{Aanalytic}, 
	via a special function $h$, see \cite[Lemma 4.2]{sadiqScherzerTamasan} for details.
	The function $h$ is defined as 
	\begin{align}\label{hDefn}
		h(z,\btheta) := Da(z,\btheta) -\frac{1}{2} \left( I - \i H \right) Ra(z\cdot \btheta^{\perp}, \btheta^{\perp}),
	\end{align} where $\btheta^\perp$ is  the counter-clockwise rotation of $\btheta$ by $\pi/2$,
	$Ra(s, \btheta^{\perp}) = \ds \int_{-\INF}^{\INF} a\left( s \btheta^{\perp} +t \btheta \right)dt$ is the Radon transform in $\BR^2$  of the attenuation $a$,
	$Da(z,\btheta) =\ds \int_{0}^{\INF} a(z+t\btheta)dt$ is the divergent beam transform of the attenuation $a$, 
	and $\ds H h(s) = \ds \frac{1}{\pi} \int_{-\INF}^{\INF} \frac{h(t)}{s-t}dt $ is the classical Hilbert transform \cite{muskhellishvili}, 
	taken in the first variable and evaluated at $s = z \cdotp \btheta^{\perp}$. 
	The function $h$ appeared first in \cite{nattererBook} and enjoys the crucial property of having vanishing negative Fourier modes yielding the expansions
	\begin{align}\label{ehEq}
		e^{- h(z,\btheta)} := \sum_{k=0}^{\INF} \alpha_{k}(z) e^{\i k\tta}, \quad e^{h(z,\btheta)} := \sum_{k=0}^{\INF} \beta_{k}(z) e^{\i k\tta}, \quad (z, \btheta) \in \ol\OM \times \sph.
	\end{align}
	Using the Fourier coefficients of  $e^{\pm h}$,  define the operators $e^{\pm G} \bu $ component-wise for each $n \leq 0$, by 
	\begin{align}\label{eGop}
		(e^{-G} \bu )_n &= (\balpha \ast \bu)_n = \sum_{k=0}^{\infty}\alpha_{k} u_{n-k}, \quad \text{and} \quad 
		(e^{G} \bu )_n = (\bbeta \ast \bu)_n = \sum_{k=0}^{\infty}\beta_{k} u_{n-k}, \quad \text{where} \\ \nonumber
		&\ol \OM \ni z\mapsto \balpha(z) := \langle \alpha_{0}(z), \alpha_{1}(z),  ... , \rangle,  
		\quad \ol \OM \ni z\mapsto \bbeta(z) := \langle \beta_{0}(z), \beta_{1}(z),  ... , \rangle. 
	\end{align}
	We refer \cite[Lemma 4.1]{sadiqScherzerTamasan} for the properties of $h$, and we restate the following result \cite[Proposition 5.2]{sadiqTamasan01} to incorporate the operators $e^{\pm G}$ notation used in here.
	\begin{prop}\cite[Proposition 5.2]{sadiqTamasan01}\label{eGprop}
		Let $a\in C^{1,\mu}(\ol \OM)$, $\mu>1/2$. Then $\balpha , \del \balpha, \bbeta , \del \bbeta \in  l^{1,1}_{\INF}(\ol \OM)$, and 
		the operators maps
		\begin{align*}
			&(i) \,e^{\pm G}:C^{\mu} (\ol\OM ; l_{\INF}) \to C^{\mu} (\ol\OM ; l_{\INF}); \; (ii) \, e^{\pm G}:C^{\mu} (\ol\OM ; l_{1}) \to C^{\mu} (\ol\OM ; l_{1}); \; (iii) \, e^{\pm G}: Y_{\mu}(\Gam) \to Y_{\mu}(\Gam).
		\end{align*}
	\end{prop}
	\begin{lemma}\cite[Lemma 4.2]{sadiqTamasan02}\label{beltrami_reduction}
		Let $a\in C^{1,\mu}(\ol \OM)$, $\mu>1/2$, and $e^{\pm G}$ be operators as defined in \eqref{eGop}. 
		
		(i) If $\bu \in C^1(\OM, l_1)$ solves $\ds \dba \bu +L^2 \del \bu+ aL\bu = \bzero$, then  $\ds \bv= e^{-G} \bu \in C^1(\OM, l_1)$ solves $\dba \bv + L^2\del \bv =\bzero$.
		
		(ii) Conversely, if $\bv \in C^1(\OM, l_1)$ solves $\dba \bv + L^2\del \bv =\bzero$, then $\ds \bu= e^{G} \bv \in C^1(\OM, l_1)$ solves $\ds \dba \bu +L^2 \del \bu+ aL\bu = \bzero$.
	\end{lemma}
	\section{Reconstruction of a sufficiently smooth linearly anisotropic source}\label{sec:reconstruct}
	For an isotropic real valued vector field $\bF= \langle F_{1}, F_{2}\rangle$, and real map $f_0$, recall the boundary value problem \eqref{TransportScatEq1}:
	\begin{equation}\label{TransportScatEq2}
		\begin{aligned} 
			&\btheta\cdot\nabla u(z,\btheta) +a(z) u(z,\btheta) - \int_{\sph} k(z,\btheta \cdot \btheta')u(z,\btheta') d\btheta' = \underbrace{f_0(z) + \btheta \cdot \bF(z)}_{f(z, \btheta)}, \; (z,\btheta)\in \OM\times\sph,\\
			&u\lvert_{\Gamma_-}=0,
		\end{aligned}
	\end{equation}
	with an isotropic attenuation $a=a(z)$, and with the scattering kernel $k(z,\btheta,\btheta')=k(z,\btheta\cdot\btheta')$ depending polynomially on the angle between the directions,
	\begin{align}\label{k}  
		k(z,\cos \tta)  = k_0(z) +2\sum_{n=1}^{M} k_{-n}(z) \cos (n \tta),
	\end{align}
	for some fixed integer $M \geq1$.
	Note that, since $k(z, \cos \tta)$  is both real valued and even in $\tta$, the coefficient $k_{-n}$ is the  $(-n)^{th}$ Fourier coefficient of $k(z,\cos(\cdot))$. Moreover $k_{-n}$ is real valued, and $k_n(z)=k_{-n}(z) =\frac{1}{2\pi} \int_{-\pi}^{\pi} k(z,\cos \theta ) e^{i n\theta}d\theta$.
	
	For the real vector field $\bF= \langle F_{1}, F_{2}\rangle$, the real map $f_0$, and $\btheta = (\cos \tta , \sin \tta) \in \sph$, a calculation shows that the linear anisotropic source 
	\begin{align} \label{Ftta_Calc}
		f(z, \btheta) = f_0(z) + \btheta \cdot \bF(z)  &= f_0(z) + \ol{f_{1}(z)} e^{\i \tta}+ f_{1}(z) e^{-\i \tta}, \quad \text{where} \; f_{1} = \left( F_{1}+ \i F_{2} \right) /2.
	\end{align} 

	We assume that the coefficients $a, k_0,k_{-1},...,k_{-M}\in C^{3}(\ol\OM)$ are such that the forward problem \eqref{TransportScatEq2} has a unique solution $u\in L^p(\Omega\times\sph)$ for any $f\in L^p(\Omega\times\sph)$, $p > 1$, see Theorem \ref{analytic_Fredholm}. Moreover, we assume also an \emph{unknown} source of a priori regularity $f \in W^{3,p}(\ol \OM; \BR)$, $p > 4$, and by  Theorem \ref{u_reg_Wp} part (iii), the solution  $u \in W^{3,p}(\OM \times \sph)$ $p > 4$.
	Furthermore, the functions $a,k$ and source  $f$  are assumed real valued, so that the solution $u$ is also real valued.  
	
	Let $u(z,\btheta) = \sum_{-\infty}^{\infty} u_{n}(z) e^{in\tta}$ be the formal Fourier series representation of the solution of \eqref{TransportScatEq2} in the angular variable $\btheta=(\cos\tta,\sin\tta)$. Since $u$ is real valued, the Fourier modes $\{u_{n}\}$  occurs in complex-conjugate pairs $u_{-n}=\ol{u_n}$, and the angular dependence is completely determined by the sequence of its nonpositive Fourier modes 
	\begin{align}\label{boldu}
		\OM \ni z\mapsto  \bu(z)&: = \langle u_{0}(z), u_{-1}(z),u_{-2}(z),... \rangle.
	\end{align}
	
	For	the derivatives  $\del,\dba$  in the spatial variable as in \eqref{dbar_op}, the advection operator $\btheta \cdot\nabla$ in \eqref{TransportScatEq2}  becomes  $\btheta \cdot\nabla=e^{-\i \tta}\dba + e^{\i \tta}\del$.  By identifying the Fourier coefficients of the same order, the equation \eqref{TransportScatEq2} reduces to the system:
	\begin{align}\label{uzero_eq}
		\overline{\del} u_{1}(z)+\del u_{-1}(z) + [a(z)-k_0(z)]u_{0}(z)&=f_0(z), \\ \label{source_syseq}
		\overline{\del} u_{0}(z)+\del u_{-2}(z) +[a(z)-k_{-1}(z)]u_{-1}(z)&= f_{1}(z), \\ \label{finsys}
		\dba u_{-n}(z) +\del u_{-n-2}(z) + [a(z)-k_{-n-1}(z)]u_{-n-1}(z) &= 0,\quad 1 \leq n\leq M-1, \\ \label{infinite_sys}
		\dba u_{-n}(z) +\del u_{-n-2}(z) + a(z)u_{-n-1}(z) &=0,\qquad n\geq M,
	\end{align}
	where $f_{1}$ as in \eqref{Ftta_Calc}.
			By Hodge decomposition \cite{vladimirBook}, any vector field $\bF= \langle F_{1}, F_{2}\rangle\in
	H^1(\OM;\mathbb{R}^2)$ decomposes into a gradient field and a 	divergence-free (solenoidal) field :
	\begin{equation}\label{hodge}
	\bF=\nabla \varphi+ \bF^s,  \quad \varphi|_{\del\OM} =0, \, \Div \bF^s =0,
	\end{equation}
	where $\varphi \in H_{0}^2(\Omega;\BR)$ and $\bF^s = \langle F_{1}^s,F_{2}^s \rangle\in H_{\Div}^1(\OM;\BR^2):= \{\bF^s \in H^1(\OM;\BR^2): \Div \bF^s =0\}$.
	
	
Note that for $f_{1} = \left( F_{1}+ \i F_{2} \right) /2$, we	 have $\ds	4 \del f_1 =  \Div \bF + \i \curl \bF,$ and  using $	8 \dba \del  f_1 = 2 \Delta f_1 = \Delta F_1 + \i \Delta F_2$, we have   	
\begin{equation}\label{eq:DelF12}
	\begin{aligned}
		\Delta F_1 &= \del_{x_1} \Div \bF - \del_{x_2} \curl \bF, \; \text{and } 
		\Delta F_2 = \del_{x_2} \Div \bF + \del_{x_1} \curl \bF. 
	\end{aligned}
\end{equation}	

Moreover, for $f_{1}^s = \left( F_{1}^s+ \i F_{2}^s \right) /2$, the Hodge decomposition \eqref{hodge} can be rewritten as 
\begin{equation} \label{hodge_complexification}
	f_1 = \bar{\partial}\varphi + f_1^s, \quad \varphi|_{\partial \Omega}=0,\quad \mathbb{R}e(\partial f_1^s) =0.
\end{equation}

	\begin{theorem}\label{Th:LuFs}
	Let $\OM\subset\mathbb{R}^2$ be a strictly convex bounded domain, and $\Gamma$ be its boundary. Consider the boundary value problem \eqref{TransportScatEq2} for some known real valued $a, k_0,k_{-1},...,k_{-M}\in C^{3}(\ol\OM)$ such that \eqref{TransportScatEq2} is well-posed. 
	If scalar and vector field sources $f_0$ and $\bF$ are real valued, $W^{3,p}( \OM; \BR)$ and $W^{3,p}( \OM; \BR^2)$-regular, respectively, with $p > 4$, 
	then $u\lvert_{\Gamma_+}$ uniquely determine  the solenoidal part $\bF^s$ in $\Omega$ and $u - u_0$ in $\OM$, where $u_0$ is the zeroth Fourier mode of $u$ in the angular variable.
\end{theorem}
\begin{proof}
	Let $u$ be the solution of the boundary value problem \eqref{TransportScatEq2} and let $\bu= \langle u_{0}, u_{-1}, u_{-2}, ... \rangle $  be the sequence valued map of its non-positive Fourier modes.
	Since the scalar field $f_0 \in W^{3,p}(\OM;\BR),\,p > 4$, and isotropic vector field $\bF \in W^{3,p}(\OM;\BR^2), \,p > 4$, then the anisotropic source $f = f_0 + \btheta \cdot \bF $ belong to $W^{3,p}(\OM \times \sph)$ for $p > 4$. By applying Theorem \ref{u_reg_Wp} (iii), we have $u \in W^{3,p}(\Omega\times\sph), \,p > 4$. Moreover, by the Sobolev embedding \cite{adam},
	$ W^{3,p}(\Omega\times\sph) \subset C^{2,\mu}(\ol \OM\times\sph)$ with 
	$\mu =1-\frac{2}{p}>\frac{1}{2}$, we have  $u\in C^{2,\mu}(\ol \OM\times\sph)$,
	and thus, by \cite[Proposition 4.1 (ii)]{sadiqTamasan01}, the sequence valued map $\bu \in Y_{\mu}(\Gam)$. 
	
		Since $\bF \in W^{3,p}(\OM;\BR^2), p >4$, then by compact imbedding of Sobolev spaces  \cite{adam},  $\bF \in H^1(\OM;\BR^2)$. By Hodge decomposition \eqref{hodge}, 	field $ \ds \bF = \nabla \varphi +\bF^s, $ with $ \varphi|_{\Gam} =0,$ and $ \Div \bF^s =0$. 
	
	We note from \eqref{infinite_sys} that the shifted sequence valued map $L^M\bu=\langle u_{-M},u_{-M-1},u_{-M-2},...\rangle$ solves 
	\begin{align}\label{Lmu_beltrami}
		\dba L^{M} \bu(z) +L^2 \del L^{M} \bu(z)+ a(z)L^{M+1}\bu(z) = \bzero,\quad z\in \OM.
	\end{align}
	
	Let $\ds \bv := e^{-G} L^M \bu $, then by Lemma \ref{beltrami_reduction}, and the fact that the operators $e^{\pm G}$ commute with the left translation, $ [e^{\pm G}, L]=0$, the sequence $\ds \bv = L^M e^{-G}  \bu$ solves $\dba \bv + L^2\del \bv =\bzero$, i.e $\bv$ is $L^2$ analytic.
	
	By \eqref{data}, the data $u|_{\Gamma_+} = g$ determines $L^M\bu$ on $\Gamma$. 
	By Proposition \ref{eGop} (iii), and  the convolution formula, traces $L^M\bu \lvert_{\Gamma}$  determines the traces $\bv \in Y_{\mu}(\Gam)$ on $\Gamma$.
	
	Since $\bv \lvert_{\Gamma}  \in  Y_{\mu}(\Gam)$ is the boundary value of an $L^2$-analytic function in $\OM$, then Theorem \ref{NecSuf_BukhgeimHilbert_Thm} (i) yields 
	\begin{align}\label{Cond1}
		[I+\i\HT] \bv \lvert_{\Gamma}={\bf {0}},
	\end{align}where $\HT$ is the Bukhgeim-Hilbert transform in \eqref{BHtransform}.
	
	From $\bv $ on $\Gam$, we use the Bukhgeim-Cauchy Integral formula \eqref{BukhgeimCauchyFormula} to 
	construct the sequence valued map $\bv $ inside $\OM$.  By Theorem \ref{BukhgeimCauchyThm} and Theorem \ref{NecSuf_BukhgeimHilbert_Thm} (ii),
	the constructed sequence valued  $\bv\in C^{1,\mu}(\OM;l_1)\cap C^{\mu}(\ol \OM;l_1)\cap C^2(\OM;l_\infty)$ is $L^2$-analytic in $\OM$.
	
	We use again the convolution formula  $\ds L^M \bu = e^{G}\bv $, and determine modes $u_{-n}$ now inside $\OM$, for $n \geq M$. In particular, we recover modes $u_{-M-1}, u_{-M}\in C^2(\OM)$.
	
	Recall that the modes $u_{-1}, u_{-2}, \cdots, u_{-M}, u_{-M-1}$ satisfy
	\begin{subequations} \label{Dirichlet_inhomDbar}
		\begin{align} \label{dbaU_eq}
			\dba u_{-M+j} &= -\del u_{-M+j-2} - \left[ (a- k_{-M+j-1}) u_{-M+j-1} \right], \quad 1\leq j \leq M-1,\\ 
			\label{u_Lambda}  u_{-M+j} \lvert_{\Gamma} &= g_{-M+j}.
		\end{align}
	\end{subequations}

	By applying $4 \del $ to \eqref{dbaU_eq}, the mode $u_{-M+1}$ (for $ j=1$) is then the solution to the Dirichlet problem for the  Poisson equation 
	\begin{subequations} \label{Transport_UM}
		\begin{eqnarray} \label{Poisson_UM}
			\Delta u_{-M+1} &=& -4\del^2 u_{-M-1} -4\del \left[ (a- k_{-M}) u_{-M} \right],\\  
			\label{UM_Gam} 
			u_{-M+1} \lvert_{\Gam} &=& g_{-M+1},
		\end{eqnarray} where the right hand side of \eqref{Transport_UM} is known.
	\end{subequations}
	
	We solve repeatedly \eqref{Transport_UM} for $j=2,...,M-1$ in \eqref{Dirichlet_inhomDbar},  to  recover  
	\begin{align}\label{recoveru_Ms1}
		u_{-M+1}, u_{-M+2}, \cdots, u_{-1}, \quad \text { in } \OM.
	\end{align}
	From $L^M \bu$ and \eqref{recoveru_Ms1}, $L \bu = \langle u_{-1}, u_{-2}, ... \rangle $ is determined in $\OM$. Thus $u - u_0$ is determined in $\OM$.
	
	Since
	$u_0,u_{-1},u_{-2}\in C^2(\OM)$, we can take $4\partial$ on both
	sides of the equation \eqref{source_syseq} to get
	\begin{align}\label{takede}
		\Delta u_0+ 4\del^2 u_{-2}+4\del ( [a-k_{-1}]u_{-1} )=4\del f_1 = 	 \Div \bF + \i \curl \bF.
	\end{align}
	Moreover, since $u_0$ is real valued and $\Div \bF= \Delta \varphi$, by equating the real part in \eqref{takede} we get
	the boundary value problem:
	\begin{subequations} \label{reconst_uzero1}
		\begin{eqnarray} \label{Poisson_uzero1}
			\Delta( u_0-\varphi) &=& - 4 \re \left[\del^2 u_{-2} +\del ( [a-k_{-1}]u_{-1} ) \right],\\  
			\label{uzero_Gam1} 
			(u_0-\varphi) \lvert_{\Gam} &=& g_0,
		\end{eqnarray} where the right hand side of \eqref{reconst_uzero1} is known.
	\end{subequations}\\
	Thus, real valued $\ds  (u_0-\varphi)$ is recovered in $\OM$, by solving the Dirichlet problem for the above  Poisson equation \eqref{reconst_uzero1}.{\Large }

	From \eqref{source_syseq} and  using $f_{1}^s = f_{1}  - \bar{\partial} \varphi$	
	from \eqref{hodge_complexification}, we get 
	\begin{align}\label{eq:f1s}
		f_1^s := \overline{\del}(u_{0}(z)-\varphi(z))+\del u_{-2}(z) +[a(z)-k_{-1}(z)]u_{-1}(z), \quad z \in \OM,
	\end{align}
	with $f_1^s$ satisfying	$\mathbb{R}e(\partial f_1^s) =0.$	
	
	Thus, the solenoidal part $\bF^s = \langle 2 \re {f_1^s}, 2\im{f_1^s} \rangle$, of the vector field $\bF$ is recovered in $\OM$.
\end{proof}

If we know apriori that the vector field $\bF$ is incompressible (i.e divergenceless), then we can  reconstruct both scalar field source $f_0$ and vector field source $\bF $ in $\OM$. 
\begin{theorem}\label{main_divfree}
	Let $\OM\subset\mathbb{R}^2$ be a strictly convex bounded domain, and $\Gamma$ be its boundary. Consider the boundary value problem \eqref{TransportScatEq2} for some known real valued $a, k_0,k_{-1},...,k_{-M}\in C^{3}(\ol\OM)$ such that \eqref{TransportScatEq2} is well-posed. 
	If the unknown scalar field source $f_0$ and divergenceless vector field sources $\bF$ are real valued, $W^{3,p}( \OM; \BR)$ and $W^{3,p}( \OM; \BR^2)$-regular, respectively, with $p > 4$, 
	then the data $g_{\scaleto{f_0,\bF}{4.75pt}}$  defined in \eqref{data} uniquely determine both $f_0$ and $\bF$ in $\OM$. 
\end{theorem}
\begin{proof}
	
	Let $u$ be the solution of the boundary value problem \eqref{TransportScatEq2} and let $\bu= \langle u_{0}, u_{-1}, u_{-2}, ... \rangle $  be the sequence valued map of its non-positive Fourier modes, 
	Since the isotropic scalar and vector field $f_0 \in W^{3,p}(\OM;\BR)$, and $\bF \in W^{3,p}(\OM;\BR^2)$ respectively for $p > 4$, then the anistropic source $f = f_0 + \btheta \cdot \bF \in W^{3,p}(\OM \times \sph)$ and by applying Theorem \ref{u_reg_Wp} (iii), we have $u \in W^{3,p}(\Omega\times\sph)$. By the Sobolev embedding \cite{adam},
	$ W^{3,p}(\Omega\times\sph) \subset C^{2,\mu}(\ol \OM\times\sph)$ with 
	$\mu =1-\frac{2}{p}>\frac{1}{2}$, we have  $u\in C^{2,\mu}(\ol \OM\times\sph)$,
	and thus, by \cite[Proposition 4.1 (ii)]{sadiqTamasan01},  $\bu \in Y_{\mu}(\Gam)$. 
	
	Since $\bF \in W^{3,p}(\OM;\BR^2), p >4$, then by compact imbedding of Sobolev spaces  \cite{adam},  $\bF \in H^1(\OM;\BR^2)$. By Hodge decomposition \eqref{hodge}, 	field $ \ds \bF = \nabla \varphi +\bF^s, $ with $ \varphi|_{\Gam} =0,$ and $ \Div \bF^s =0$. 
	
	If we know apriori that the vector field $\bF$ is incompressible (i.e divergenceless $\nabla \cdot \bF = 0$).  Then $ \Laplace \varphi = \Div \bF = 0$ and $\varphi|_{\del\OM} =0$  implies $\varphi \equiv 0$ inside $\OM$. Thus, vector field  $\bF = \bF^s$ inside $\Omega$.
	
	By Theorem \ref{Th:LuFs}, the data $u\lvert_{\Gamma_+}= g_{\scaleto{f_0,\bF}{4.75pt}}$
	uniquely determine the solenoidal field $\bF^s= \bF$ in $\Omega$ by equation \eqref{eq:f1s} with $\varphi \equiv 0$, and  the sequence valued map  $L\bu= \langle u_{-1}, u_{-2}, ... \rangle $ in $\OM$. Moreover, the  real valued mode $\ds u_0$ is also then recovered (with $\varphi \equiv 0$  ) in $\OM$, by solving the Dirichlet problem for the Poisson equation \eqref{reconst_uzero1}.


%
	
	Thus, from modes $ u_{-1}$ and $u_0$, the scalar field $\displaystyle f_0$ is recovered in $\OM$ by 
	\begin{align}\label{reconst_fzero_div}
	f_0: = 	2 \re [\del u_{-1}] + [a-k_0]u_{0}.
	\end{align}
	
\end{proof}

	In the radiative transport literature, the attenuation coefficient $a = \sigma_a + \sigma_s$, where $\sigma_a $ represents pure loss due to absorption and 
	$\sigma_s(z) = \frac{1}{2\pi} \int_{0}^{2\pi} k(z,\theta) d \theta = k_0(z)$ is the isotropic part of scattering kernel.
	We consider the subcritical region: 
	\begin{align} \label{sub_cond}
		\sigma_a : = a - k_0 \geq \delta >0, \quad \text{for some positive constant } \delta. 
	\end{align}
	\begin{remark}
		In addition to the hypothesis to Theorem \ref{Th:LuFs}, if we assume that coefficients $a, k_0$ satisfies \eqref{sub_cond}, then in the region $\{z\in\OM :~f_0(z)=0\}$, one can recover explicitly the entire vector field  $\bF= \langle 2 \re {f_1}, 2\im{f_1} \rangle $. 
		Indeed,  the equation (\ref{uzero_eq})  gives $u_0=-2\re(\del u_{-1})/\sigma_a$ and,
		following (\ref{source_syseq}), 
		the vector field $\bF$ can be recovered by the formula
		\begin{equation}
		f_1=\del u_{-2}+[a-k_{-1}] u_{-1}-2\overline{\del} \left(\frac{\re(\del
				u_{-1})}{\sigma_a}\right).
		\end{equation}
	\end{remark}
%
	%
	%



    Next, we show that one can also determine both scalar field $f_0$ and vector field $\bF$, if one has the additional data $ g_{\scaleto{f_0,\bzero}{4.75pt}}$ (or $ g_{\scaleto{0,\bF}{4.75pt}}$) information, instead of $\bF$ being incompressible as in Theorem \ref{main_divfree}.
    
	
	\begin{theorem}\label{Th_2data}
		Let $\OM\subset\mathbb{R}^2$ be a strictly convex bounded domain, and $\Gamma$ be its boundary. Consider the boundary value problem \eqref{TransportScatEq2} for some known real valued $a, k_0,k_{-1},...,k_{-M}\in C^{3}(\ol\OM)$ such that \eqref{TransportScatEq2} is well-posed. 
		If the unknown scalar field source $f_0$ and vector field source $\bF$ are real valued, $W^{3,p}( \OM; \BR)$ and $W^{3,p}( \OM; \BR^2)$-regular, respectively, with $p > 4$, 
		and coefficients $a, k_0$ satisfying \eqref{sub_cond},  then the data $g_{\scaleto{f_0,\bF}{4.75pt}}$ and $ g_{\scaleto{f_0,\bzero}{4.75pt}}$ defined in \eqref{data} uniquely determine both $f_0$ and $\bF$ in $\OM$. 
	\end{theorem}
	\begin{proof}
		Let $u$ be the solution of the boundary value problem \eqref{TransportScatEq2} and let $\bu= \langle u_{0}, u_{-1}, u_{-2}, ... \rangle $  be the sequence valued map of its non-positive Fourier modes.
		Since the scalar field $f_0 \in W^{3,p}(\OM;\BR),\,p > 4$, and isotropic vector field $\bF \in W^{3,p}(\OM;\BR^2), \,p > 4$, then the anisotropic source $f = f_0 + \btheta \cdot \bF $ belong to $W^{3,p}(\OM \times \sph)$ for $p > 4$. By applying Theorem \ref{u_reg_Wp} (iii), we have $u \in W^{3,p}(\Omega\times\sph), \,p > 4$. Moreover, by the Sobolev embedding,
		$ W^{3,p}(\Omega\times\sph) \subset C^{2,\mu}(\ol \OM\times\sph)$ with 
		$\mu =1-\frac{2}{p}>\frac{1}{2}$, we have  $u\in C^{2,\mu}(\ol \OM\times\sph)$,
		and thus, by \cite[Proposition 4.1 (ii)]{sadiqTamasan01}, the sequence valued map $\bu \in Y_{\mu}(\Gam)$. 
		
		We consider the boundary value problems 
		\begin{align}\label{bvp_v} 
			\btheta \cdot \nabla v +av -Kv&= f_0, \qquad  \text{subject to } \quad v\lvert_{\Gamma_-}=0,   \quad  v\lvert_{\Gamma_+}=g_{\scaleto{f_0,\bzero}{4.75pt}}, \quad \text{and}	\\ 
		\label{bvp_w} 
			\btheta \cdot \nabla w +aw -Kw&= \btheta \cdot \bF, \quad  \text{subject to } \quad w\lvert_{\Gamma_-}=0,   \quad 
			w\lvert_{\Gamma_+}=\widetilde{g}:= g_{\scaleto{f_0,\bF}{4.75pt}}- g_{\scaleto{f_0,\bzero}{4.75pt}}.
		\end{align}
		Then $u = v+w$ satisfy the  boundary value problem \eqref{TransportScatEq2}.
		
		We consider first the boundary value problem \eqref{bvp_v}, and will reconstruct the scalar field $f_0$ from the  given boundary data $g_{\scaleto{f_0,\bzero}{4.75pt}}$. 
		
		If $\ds \sum_{n \in \BZ} v_{n}(z) e^{\i n\tta}$ is the Fourier series  expansion in the angular variable $\btheta$ of a solution $v$ of boundary value problem  \eqref{bvp_v}, then,  by identifying the Fourier modes of the same order, the equation in \eqref{bvp_v} reduces to the system:
		\begin{align}\label{vzero_eq1}
			\overline{\del} \ol{v_{-1}}(z)+\del v_{-1}(z) + [a(z)-k_0(z)]v_{0}(z)&=f_0(z), \\ 
			\label{v_finsy1s}
			\dba v_{-n}(z) +\del v_{-n-2}(z) + [a(z)-k_{-n-1}(z)]v_{-n-1}(z) &= 0,\quad 0 \leq n\leq M-1, \\ \label{v_infinite_sys1}
			\dba v_{-n}(z) +\del v_{-n-2}(z) + a(z)v_{-n-1}(z) &=0,\qquad n\geq M,
		\end{align}
		and let $\bv= \langle v_{0}, v_{-1}, v_{-2}, ... \rangle $  be the sequence valued map of its non-positive Fourier modes.
		
		By Theorem \ref {Th:LuFs}, from data $g_{\scaleto{f_0,\bzero}{4.75pt}}$,  the sequence $ L\bv = \langle v_{-1}, v_{-2}, ... \rangle$ is determined in $\OM$. Moreover, 
		as \eqref{v_finsy1s} holds also for $n=0$ ($f_1=0$ in this case), the mode $v_0$ is also determined in $\OM$ by  solving the Dirichlet problem for the  Poisson equation 
		\begin{subequations} \label{Transport_v0}
			\begin{eqnarray} \label{Poisson_v_0}
			\Delta v_{0} &=& -4\del^2 v_{-2} -4\del \left[ (a- k_{-1}) v_{-1} \right],\\  
			\label{v0_Gam} 
			v_{0} \lvert_{\Gam} &=& g_{0},
			\end{eqnarray} where the right hand side of \eqref{Transport_v0} is known.
		\end{subequations} 
		
		Thus, using modes $v_0$ and $v_{-1}$, the isotropic scalar source $f_0$ is recovered in $\OM$ by 
		\begin{align} \label{fsource_Scatpoly}
			f_{0}(z) = 2 \re \left(\del v_{-1}(z)\right) + \left( a(z)-k_0(z) \right) v_0(z),\quad z\in\OM.
		\end{align}
		
		Next, we consider the boundary value problem \eqref{bvp_w}, and will reconstruct the vector field $\bF$ from the  given boundary data $\widetilde{g} = g_{\scaleto{f_0,\bF}{4.75pt}}- g_{\scaleto{f_0,\bzero}{4.75pt}}$. 
		
		If $\ds \sum_{n \in \BZ} w_{n}(z) e^{\i n\tta}$ is the Fourier series  expansion in the angular variable $\btheta$ of a solution $w$ of the boundary value problem \eqref{bvp_w}, then,  by identifying the Fourier modes of the same order, the equation in \eqref{bvp_w} reduces to the system:
		\begin{align}\label{wzero_eq11}
			\overline{\del} \ol{w_{-1}}(z)+\del w_{-1}(z) + [a(z)-k_0(z)]w_{0}(z)&= 0, \\ \label{w_syseq1}
			\overline{\del} w_{0}(z)+\del w_{-2}(z) + [a(z)-k_{-1}(z)]w_{-1}(z)&=\left( F_{1}(z)+\i F_{2}(z)\right) /2, \\ \label{w_finsy1s}
			\dba w_{-n}(z) +\del w_{-n-2}(z) + [a(z)-k_{-n-1}(z)]w_{-n-1}(z) &= 0,\quad 1 \leq n\leq M-1, \\ \label{w_infinite_sys1}
			\dba w_{-n}(z) +\del w_{-n-2}(z) + a(z)w_{-n-1}(z) &=0,\qquad n\geq M,
		\end{align} and let $\bw= \langle w_{0}, w_{-1}, w_{-2}, ... \rangle $  be the sequence valued map of its non-positive Fourier modes.
		
			By Theorem \ref {Th:LuFs}, from data $\widetilde{g}$,  the sequence $ L\bw=  \langle w_{-1}, w_{-2}, ... \rangle $ is determined in $\OM$. 
			
		Using the subcriticality condition \eqref{sub_cond}: $\ds 
		\sigma_a (z) =a(z)-k_{0}(z)  >0, $ 
		and equation \eqref{wzero_eq11}, we define 
		\begin{align}\label{w_zero_eq}
			w_{0}(z)&:= -\frac{2\re \del w_{-1}(z)}{a(z)-k_{0}(z)} = -\frac{2\re \del w_{-1}(z)}{\sigma_a (z)}.
		\end{align}
		
		The real valued vector field $\bF = \langle 2 \re {f_1}, 2\im{f_1} \rangle$ is recovered in $\OM$ by 
		\begin{align} \label{fsource_Scatpoly_w}
			f_{1}(z) = \overline{\del} w_{0}(z) + \del w_{-2}(z) + \left[ a(z)-k_{-1}(z) \right] w_{-1}(z),\quad z\in\OM.
		\end{align}
	\end{proof}
\vspace{-0.5cm}
	\section{When can the data coming from two sources be mistaken for each other ?} \label{sources_mistaken}
	
	In this section we will address when the data coming from two different linear anisotropic field sources can be mistaken.

	In the theorem below the data are assuming the same attenuation $a$ and scattering coefficient $k$.
	\begin{theorem}\label{Isotropic_Anisotropic_datarelation}
	
	(i) Let $a\in C^{3}(\ol\OM)$,  $k\in C^{3}(\ol\OM \times \sph)$ be real valued, 
	with $\sigma_a = a -k_0 >0$, and $f_0, \widetilde{f} \in W^{3,p}( \OM)$, $p > 4$ be real valued with $
	\displaystyle( f_0 -\widetilde{f}) / \sigma_a  \in C_{0}(\ol \OM)$. Then
	$ \bF:= \displaystyle \widetilde{\bF}+\nabla \left ( \frac{f_0 -\widetilde{f}}{\sigma_a} \right )$ is a real valued vector field such that the data $g_{f_0,\bF}$ coming from the linear anisotropic source  $f_0 + \btheta \cdot \bF$,  is the same as data $g_{\widetilde{f}, \tilde{\bF}}$ coming from a different linear anisotropic source $\widetilde{f} + \btheta \cdot \widetilde{\bF}:$
	\begin{align*}
	g_{f_0,\widetilde{\bF}+\nabla \left ( \frac{f_0 -\widetilde{f}}{\sigma_a} \right )}= g_{\widetilde{f}, \widetilde{\bF}} .
	\end{align*}

	(ii) Let $a, k_0,k_{-1},...,k_{-M}\in C^{3}(\ol\OM)$  be real valued with $\sigma_a = a -k_0 >0$.
	 Assume that there are real valued linear anisotropic sources  $f_0 + \btheta \cdot \bF$ and  $\widetilde{f} + \btheta \cdot \widetilde{\bF}$, with isotropic fields $f_0,\widetilde{f}  \in W^{3,p}( \OM)$, $p > 4$, and vector fields $\bF,\widetilde{\bF} \in W^{3,p}(\OM;\BR^2)$, $p > 4$. 
	 If the data $g_{f_0,\bF}$ of the linear anisotropic source  $f_0 + \btheta \cdot \bF$ equals the data $g_{\widetilde{f}, \widetilde{F}}$ of the linear anisotropic source  $\widetilde{f} + \btheta \cdot \widetilde{\bF}$.
	   Then  
	$\bF =\widetilde{\bF} + \displaystyle \nabla \left ( \frac{f_0 - \widetilde{f}}{\sigma_a} \right )$.
\end{theorem}

\begin{proof}
	(i) Assume $g_{\widetilde{f}, \widetilde{\bF}}$ is the data of some real valued anisotropic source $\widetilde{f} + \btheta \cdot \widetilde{\bF}$, i.e., it is the trace on $\Gam \times \sph$ of solutions $w$ to the stationary transport boundary value problem:
	\begin{equation}\label{transp_function} 
		\begin{aligned}
			\btheta \cdot \nabla w +aw -Kw&= \widetilde{f}+ \btheta \cdot \widetilde{\bF},\\
			w \lvert_{\Gam \times \sph} &= g_{\widetilde{f}, \widetilde{\bF}}, 
		\end{aligned} 
	\end{equation}
	where  the operator $\ds  [Kw] (z,\btheta) :=  \int_{\sph} k(z,\btheta \cdot \btheta')w(z,\btheta') d\btheta',\; $ for $ z \in \OM$  and $\btheta \in \sph$.
	
	For $\sigma_a = a-k_0$ with $\sigma_a >0$,  and isotropic real valued functions $\psi$ and $\sigma_a$, we note: 
			\begin{equation}\label{eq:Kfa}
				\begin{aligned}
					\left[K\frac{\psi}{\sigma_a}\right] (z,\btheta) &=  \int_{\sph} k(z,\btheta \cdot \btheta')\left [\frac{\psi}{\sigma_a}\right](z,\btheta') d\btheta' \\
					&=  \frac{\psi(z)}{\sigma_a(z)}\int_{\sph} k(z,\btheta \cdot \btheta') d\btheta' = \frac{\psi(z)}{\sigma_a(z)} k_0(z),
				\end{aligned} 
			\end{equation}where second equality use the fact that both $\psi$ and $\sigma_a$ are angularly independent functions.
	
	
	Let $\displaystyle u := w + (f_0 - \widetilde{f})/\sigma_a $ and $\bF :=\widetilde{\bF} +\displaystyle \nabla \left ( \frac{f_0 - \widetilde{f}}{\sigma_a} \right )$. Then
	\begin{align*}
		\btheta \cdot \nabla &u +au -Ku= \btheta\cdot\nabla\left(w+\frac{f_0 - \widetilde{f}}{\sigma_a}\right)+a\left(w+\frac{f_0 - \widetilde{f}}{\sigma_a}\right) - K\left(w + \frac{f_0 - \widetilde{f}}{\sigma_a}\right)\\
		&=\btheta\cdot \nabla w + aw-Kw - \left(\frac{a}{\sigma_a}\right) \widetilde{f}+ \left(\frac{k_0}{\sigma_a}\right)\widetilde{f}+ \left(\frac{a}{\sigma_a}\right)f_0- \left(\frac{k_0}{\sigma_a}\right)f_0 +\btheta\cdot \nabla\left(\frac{f_0 -\widetilde{f}}{\sigma_a}\right)\\
		&= \left(1-\frac{a}{\sigma_a}+ \frac{k_0}{\sigma_a}\right)  \widetilde{f} +\left(\frac{a-k_0}{\sigma_a}\right)f_0 +\btheta\cdot \left(\widetilde{\bF} +\nabla\left(\frac{f_0 -\widetilde{f}}{\sigma_a}\right)\right)= f_0 + \btheta \cdot\bF,
	\end{align*} 
	where the second equality uses the linearty of $K$ and \eqref{eq:Kfa}, the last equality uses \eqref{transp_function},  and the definition of $\bF$. Moreover, since $f_0-\widetilde{f} /\sigma_a$ vanishes on $\Gam$, we get
	\begin{align*}g_{f_0,\bF}=u\lvert_{\Gam \times \sph} = w\lvert_{\Gam\times \sph}+\left.\frac{f_0-\widetilde{f}}{\sigma_a}\right\lvert_{\Gam}= w\lvert_{\Gam\times \sph}=g_{\widetilde{f}, \widetilde{\bF}}.
	\end{align*}
	
	(ii) Let $f_0 + \btheta \cdot \bF$, be the real valued linear anisotropic source  with isotropic field $f_0 \in W^{3,p}( \OM)$, $p > 4$, and vector field $\bF \in W^{3,p}(\OM;\BR^2)$, $p > 4$. 
	If the data $g_{f_0,\bF}$ of the linear anisotropic source  $f_0 + \btheta \cdot \bF$ equals data $g_{\widetilde{f}, \widetilde{\bF}}$ of some real valued $\widetilde{f},\widetilde{\bF} \in W^{3,p}(\OM)$, $p > 4$, i.e. $$g_{\widetilde{f}, \widetilde{\bF}}=g=g_{f_0,\bF}.$$
	Then by Theorem \ref{u_reg_Wp} (iii),  there exist  $u,w\in  W^{3,p}(\Omega\times\sph)$ solutions to the corresponding transport equations
	\begin{align}\label{bvp_uw} 
		\btheta \cdot \nabla u +au -Ku&= f_0 +\btheta \cdot \bF	,\qquad  \text{and} \quad  \btheta \cdot \nabla w +aw -Kw= \widetilde{f}+ \btheta \cdot \widetilde{\bF}
	\end{align} 
	respectively, subject to
	\begin{align*}
		u \lvert_{\Gam \times \sph} = g = w \lvert_{\Gam \times \sph}.
	\end{align*}
	Moreover, by the Sobolev embedding, $u,w\in C^{2,\mu}(\ol \OM\times\sph)$ with 
	$\mu =1-\frac{2}{p}>\frac{1}{2}$, and the corresponding sequences of non-positive Fourier modes $\{u_{-n}\}_{n\geq 0}$ of $u$ satisfy
	\begin{align}\label{uzero_eq1}
		\overline{\del} \ol{u_{-1}}(z)+\del u_{-1}(z) + [a(z)-k_0(z)]u_{0}(z)&=f_0(z), \\ \label{u_syseq1}
		\overline{\del} u_{0}(z)+\del u_{-2}(z) + [a(z)-k_{-1}(z)]u_{-1}(z)&= \left( F_{1}(z)+i F_{2}(z)\right) /2, \\ \label{u_finsy1s}
		\dba u_{-n}(z) +\del u_{-n-2}(z) + [a(z)-k_{-n-1}(z)]u_{-n-1}(z) &= 0,\quad 1 \leq n\leq M-1, \\ \label{u_infinite_sys1}
		\dba u_{-n}(z) +\del u_{-n-2}(z) + a(z)u_{-n-1}(z) &=0,\qquad n\geq M,
	\end{align}
	whereas the non-positive Fourier modes $\{w_{-n}\}_{n\geq 0}$ of $w$ satisfy
	\begin{align}\label{wzero_eq1}
		\overline{\del} \ol{w_{-1}}(z)+\del w_{-1}(z) + [a(z)-k_0(z)]w_{0}(z)&= \widetilde{f}(z), \\ \label{w_syseq1}
		\overline{\del} w_{0}(z)+\del w_{-2}(z) + [a(z)-k_{-1}(z)]w_{-1}(z)&= \left( \widetilde{F_{1}}(z)+i\widetilde{F_{2}}(z)\right) /2, \\ \label{w_finsy1s}
		\dba w_{-n}(z) +\del w_{-n-2}(z) + [a(z)-k_{-n-1}(z)]w_{-n-1}(z) &= 0,\quad 1 \leq n\leq M-1, \\ \label{w_infinite_sys1}
		\dba w_{-n}(z) +\del w_{-n-2}(z) + a(z)w_{-n-1}(z) &=0,\qquad n\geq M.
	\end{align}
	Furthermore, by \cite[Proposition 4.1 (ii)]{sadiqTamasan01}, the corresponding sequence valued   $\bu = \langle u_0, u_{-1}, \cdots \rangle \in Y_{\mu}(\Gam)$, and  $\bw = \langle w_0, w_{-1}, w_{-2},\cdots \rangle \in Y_{\mu}(\Gam)$ with $\mu >\frac{1}{2}$.
	
	Since the boundary data $g$ is the same $u|_{\Gam\times\sph}=w|_{\Gam\times\sph}$, we also have
	\begin{align}\label{bddIdentity}
		\bu \lvert_\Gam = \bg =\bw \lvert_\Gam.
	\end{align}
	
	We claim that the systems \eqref{u_infinite_sys1} and \eqref{w_infinite_sys1} subject to the identity \eqref{bddIdentity}, 
	yield
	\begin{align}\label{un=wn_n>M}
		L^M\bu(z) = L^M\bw(z),\quad z\in\OM.
	\end{align}
    The shifted sequence valued maps $L^M\bu= \langle u_{-M},u_{-M-1},...\rangle $, and
    $L^M\bw=(w_{-M},w_{-M-1},...)$, respectively, solves systems \eqref{u_infinite_sys1}, and  \eqref{w_infinite_sys1}.  Then the sequence valued map $L^M\bv:=\langle v_{-M}, v_{-M-1}, ...\rangle$, and $L^M\brho = \langle \rho_{-M},\rho_{-M-1},...\rangle $ are defined by 
\begin{align}\label{recover_LMv_rho}
	L^M\bv &=  e^{G}\B \left (L^M e^{-G} \bg\right), \quad L^M\brho  = e^{G}\B \left (L^M e^{-G} \bg \right),
\end{align} where $\B$ is 	the Bukhgeim-Cauchy operator  in \eqref{BukhgeimCauchyFormula}, and  $e^{\pm G}$ are  the operators in \eqref{eGop}.

%
	In particular, $L^M\bv$ and $ L^M\brho$ are $L^2$-analytic, and coincide at the boundary $\Gamma$.
	By uniqueness of $L^2$-analytic functions with a given trace, they coincide inside:
	\begin{align}\label{vn=pn_n>M}
		L^M\bv(z) =L^M\brho(z),\quad \text{for} \; z\in \OM.
	\end{align}
	Using the operator $e^{- G}$ in \eqref{eGop}, we conclude that
	\begin{align*}
		L^M\bu(z) &=  e^{-G}L^M \bv(z), \quad  L^M\bw(z) =  e^{-G}L^M \brho(z), \quad z \in \OM.
	\end{align*}Thus \eqref{un=wn_n>M} holds.
	
	Subjecting \eqref{u_finsy1s} and \eqref{w_finsy1s} to the boundary conditions \eqref{bddIdentity}, we claim that 
	\begin{align}\label{un=wn_n<M}
		u_{-n}(z) = w_{-n}(z), \quad z\in\OM, \quad \text{for all }\; 1\leq n\leq M-1.
	\end{align}
	
	Define 
	\begin{align}\label{psi_defn}
		\psi_{-j}:=u_{-j}- w_{-j}, \quad \text{for} \;  j \geq 1,
	\end{align}
	and note that by \eqref{un=wn_n>M}, we have  
	\begin{align}\label{psi_>M}
		\psi_{-j}=0, \quad \text{for} \; j \geq M.
	\end{align}
	By subtracting  \eqref{w_finsy1s} from \eqref{u_finsy1s}, and using \eqref{psi_defn}, and \eqref{bddIdentity}, we have 
	\begin{subequations} \label{Dirichlet_inhomDbar_psi}
		\begin{align} \label{dbapsi_eq}
			\dba \psi_{-M+j} &= -\del \psi_{-M+j-2} - \left[ (a- k_{-M+j-1}) \psi_{-M+j-1} \right], \quad 1\leq j \leq M-1,\\ 
			\label{psi_gam}  \psi_{-M+j} \lvert_{\Gamma} &= 0.
		\end{align}
	\end{subequations} 
	For the mode $\psi_{-M+1}$ (when $j=1$), the right hand side of \eqref{dbapsi_eq} contains modes $\psi_{-M-1}$ and $\psi_{-M}$ which are both zero by \eqref{psi_>M}. Thus, the mode $\psi_{-M+1}\equiv0$ is the unique solution to the Cauchy problem for the $\dba$-equation, 
	\begin{subequations} \label{Cauchyproblem}
		\begin{align} \label{dba_Cauchyproblem}
			\dba \Psi &= 0, \quad \text{in} \; \OM,\\ \label{u_Lambda_data}
			\Psi &= 0, \quad \text{on} \; \Gamma.
		\end{align}
	\end{subequations}  
	We then solve repeatedly \eqref{Dirichlet_inhomDbar_psi} starting for $j=2,...,M-1$, where the right hand side of \eqref{dbapsi_eq} in each step is zero, yielding the Cauchy problem \eqref{Cauchyproblem} for each subsequent modes, and thus, resulting in the  recovering of the modes $\psi_{-M+1} =\psi_{-M+2}= \cdots \psi_{-2}=\psi_{-1} \equiv 0$ in $\OM$. Therefore, establishing \eqref{un=wn_n<M}.

	By subtracting \eqref{wzero_eq1} from \eqref{uzero_eq1}, we obtain
	$\ds 		(a-k_0)(u_0-w_0)= f_0-\widetilde{f}$, and thus  $\ds 	u_0-w_0=\frac{f_0-\widetilde{f}}{\sigma_a}$.
   	Moreover, by subtracting \eqref{w_syseq1} from \eqref{u_syseq1} and using \eqref{un=wn_n>M} and \eqref{un=wn_n<M}, yields $\ds 2\ol\del(u_0-w_0)= (F_1-\widetilde{F_1})+\i (F_2-\widetilde{F_2}).$
   Since both $u_0$ and $w_0$
   are real valued we see that $\ds 
   \bF-\widetilde{\bF}=\nabla(u_0-w_0)$, and we have 
	\begin{align*}
		\bF =\widetilde{\bF} + \nabla \left ( \frac{f_0-\widetilde{f}}{\sigma_a} \right ).
	\end{align*}
	
\end{proof}

\begin{remark}
Note that in Theorem \ref{Isotropic_Anisotropic_datarelation}(i), the assumption on scattering kernels of finite Fourier content in the angular variable is not assumed, and the result holds for a  general scattering kernels which depends polynomially on the angle between the directions.
\end{remark}
	
%
%
%
	
	
	\section*{Acknowledgment}
	The work of D.~Omogbhe and K.~ Sadiq  were supported by the Austrian Science Fund (FWF), Project P31053-N32 and by the FWF Project F6801–N36 within the Special Research Program SFB F68 “Tomography Across the Scales”. 
	


\begin{thebibliography}{99}
		
		\bibitem{adam} R.~ Adam, Sobolev spaces, Academic Press, New York,  1975.
		
		\bibitem{anikonov02}  D.~S.~Anikonov, A.~E.~Kovtanyuk and I.~V.~Prokhorov, \textit{Transport Equation and Tomography}, Inverse and Ill-Posed Problems Series, 30, Brill Academics, (2002).
		
		
		
		
		
		\bibitem{ABK} E.~V.~Arbuzov, A.~L.~Bukhgeim and S.~G.~Kazantsev, \textit{Two-dimensional tomography problems and the theory of A-analytic functions}, Siberian Adv.\ Math., \textbf{8} (1998), 1--20.
		
		
		
		
		
		
		
		
		
		
		
		\bibitem{balTamasan07} G.~Bal and A.~Tamasan, \textit{Inverse source problems in transport equations}, SIAM J.\ Math.\ Anal., \textbf{39} (2007), 57--76.
		
		
		
		
		
		
		\bibitem{bukhgeimBook} A.~L.~Bukhgeim, \textit{Inversion Formulas in Inverse Problems}, chapter in Linear Operators and Ill-Posed Problems by M.~M.~Lavrentiev and L.~Ya.~Savalev, Plenum, New York, 1995.
		
		\bibitem{cercignani} C.~ Cercignani,  The Boltzmann Equation and Its Applications, Berlin: Springer-Verlag; 1988.
		
		
		
		\bibitem{chandrasekhar} S.~Chandrasekhar, Radiative Transfer, Dover Publ., New York, 1960.
		
		
		\bibitem{choulliStefanov99} M.~Choulli and P.~Stefanov, \textit{An inverse boundary value problem for the stationary transport equation}, Osaka J.\ Math., \textbf{36} (1999), 87--104.
		
		\bibitem{dautrayLions4} R.~Dautray and J-L.~Lions, Mathematical analysis and numerical methods for science and technology, Vol.~4.,  Springer-Verlag, Berlin, 1990.
		
		
		
		\bibitem{dunfordScwartz}  N.~Dunford and J.~Schwartz, {Linear Operators, Part I: General Theory}, Wiley, Interscience Publ., New York, 1957.
		
		\bibitem{eggerSchlottbom} H.~Egger and M.~Schlottbom,
		\textit{An $L^p$ theory for stationary radiative transfer}, Applicable Analysis, \textbf{93(6)} (2014), 1283--1296.
		
		
		
		\bibitem{finch} D.~V.~Finch, \textit{The attenuated x-ray transform: recent developments}, in Inside out: inverse problems and applications, Math.\ Sci.\ Res.\ Inst.\ Publ., \textbf{47}, Cambridge Univ.\ Press, Cambridge, 2003, 47--66.
		
		
		\bibitem{fujiwaraSadiqTamasan19} H.~Fujiwara, K.~Sadiq and A.~Tamasan, \textit{A Fourier approach to the inverse source problem in an absorbing and anisotropic scattering medium}, Inverse Problems {\bf 36(1)}:015005 (2019).
		
		
		
		
		\bibitem{fujiwaraSadiqTamasan21b} H.~Fujiwara, K.~Sadiq and A.~Tamasan, \textit{ A two dimensional source reconstruction method in radiative transport using boundary data measured on an arc},   Inverse Problems \textbf{37(11)} (2021), 19pp.
		
		
		
		\bibitem{han} W.~Han, J.~A.~Eichholz, J.~Huang, \textit{RTE-based bioluminescence tomography: A theoretical study}, Inverse~Probl.~Sci.~Eng., \textbf{19(4)} (2011), 435--459.
		
		
		\bibitem{klose} A.~D.~Klose, V.~Ntziachristos, and A.~H.~Hielscher, \textit{The inverse source problem based on the radiative transfer equation in optical molecular imaging}, J.~Comput.~Phys., \textbf{202} (2005), 323--345.	
		
		
		
%
		
		
		
		
		
		
		
		
		
		
		
		
		
		
		
		
		
		
		
		
		
		
%
		
		\bibitem{mokhtar} M.~Mokhtar-Kharroubi, \textit{Mathematical topics in neutron transport theory}, World Scientific, Singapore, 1997.
		
		
		
		
		
		
		\bibitem{monardBalPreprint} F.~Monard and G.~Bal, \textit{Inverse source problems in transport via attenuated tensor tomography}, arXiv:1908.06508v1. 
		
		
		
		
		
		
		
		\bibitem{muskhellishvili} N.~I.~Muskhelishvili, Singular Integral Equations, Dover, New York, 2008.

		\bibitem{nattererBook} F.~Natterer, {The mathematics of computerized tomography}, Wiley, New York, 1986.
		
		
		
		
		
		
	\bibitem {nattererWubbeling} F.~Natterer and F.~W\"{u}bbeling,  Mathematical methods in image reconstruction. \textit{SIAM Monographs on Mathematical Modeling and Computation}, SIAM, Philadelphia, PA, 2001
		
		
		
		
		
		
		
	
		
		
		
		
		
		
		\bibitem{sadiqTamasan01} K.~Sadiq and A.~Tamasan, \textit{On the range of the attenuated Radon transform in strictly convex sets}, Trans.\ Amer.\ Math.\ Soc., \textbf{367(8)} (2015), 5375--5398.
		
		\bibitem{sadiqTamasan02} K.~Sadiq and A.~Tamasan, \textit{On the range characterization of the two dimensional attenuated Doppler transform},  SIAM J.\ Math.\ Anal., \textbf{47(3)} (2015), 2001--2021.
		
	\bibitem{sadiqScherzerTamasan} K.~Sadiq, O.~Scherzer, and A.~Tamasan, \textit{On the X-ray transform of planar symmetric 2-tensors}, J.\ Math.\ Anal.\ Appl., \textbf{442(1)} (2016),  31--49.
		
		
%
		
		\bibitem{vladimirBook} V.~A.~Sharafutdinov, {Integral geometry of tensor fields}, VSP, Utrecht, 1994.
		
		
		
		
		
		%
%
	\bibitem {sparSLP95} G. Sparr, K. Str{\aa}hl\'en, K. Lindstr\"om, and H. W. Persson, \textit{Doppler tomography for vector fields,} Inverse Problems, \textbf{11} (1995), 1051--1061.
		
		
		
		
		
		
		
		\bibitem{stefanovUhlmann08} P.~Stefanov and G.~Uhlmann, \textit{ An inverse source problem in optical molecular imaging}, Anal.\ PDE, \textbf{1(1)} (2008), 115--126.
		
		
		
		
		
		%
		
		\bibitem{tamasan07} A.~Tamasan, \textit{Tomographic reconstruction of vector fields in variable background media}, Inverse Problems \textbf{23} (2007), 2197--2205.
%
		
		
		
%
		\bibitem{YiSanchezMccormick92} H.~C.~Yi, R.~Sanchez, and N.~J.~McCormick, 	\textit{Bioluminescence   estimation from ocean {\it in situ} irradiances}, Applied Optics, \textbf{31} (1992), 822--830.
	\end{thebibliography}
\end{document}